\documentclass[11pt,dvipdfmx,a4paper]{article}
\usepackage[margin=30mm]{geometry}
\usepackage[colorlinks=true,linkcolor=blue,citecolor=blue]{hyperref}
\usepackage{amssymb,amsthm,amsmath, amsfonts,ascmac, enumerate}
\numberwithin{equation}{section}

\theoremstyle{definition}
\newtheorem{thm}{Theorem}[section]
\newtheorem{defn}[thm]{Definition}
\newtheorem{lem}[thm]{Lemma}
\newtheorem{prop}[thm]{Proposition}
\newtheorem{rem}[thm]{Remark}

\newcommand{\R}{\mathbb{R}}   
   
\newcommand{\N}{\mathbb{N}}

\newcommand{\calH}{{\mathcal{H}}}

\newcommand{\calV}{{\mathcal{V}}}

\newcommand{\calX}{{\mathcal{X}}}

\title{On the rate of convergence to the Boolean extreme value distribution under the von Mises condition}

\author{Yuki Ueda}
\date{}

\begin{document}

\maketitle

\begin{abstract}
We investigate the rate of convergence toward the Boolean extreme value distribution, which is the universal limiting law for the normalized spectral maximum of Boolean independent and identically distributed positive operators, under the von Mises condition.
\end{abstract}


\section{Introduction}

Extreme value theory has long served as a foundational framework for describing the asymptotic behavior of the maximum of a sequence of independent and identically distributed (i.i.d.) random variables. Over time, it has evolved into a major branch of probability and statistics, playing a central role particularly in the study of large observations and rare events. One of the most important results in extreme value theory is Fisher-Tippett-Gnedenko theorem (\cite{FT28}, \cite{G43}, \cite{R}), which states that, under appropriate normalization, the distribution of the maximum of i.i.d. random variables converges to one of the {\it extreme value distributions}:
$$
\Phi_\alpha(x)=\begin{cases}
\exp(-x^{-\alpha})\mathbf{1}_{(0,\infty)}(x), & \alpha>0 \quad \text{(Fr\'{e}chet distribution)};\\
\exp(-(-x)^{-\alpha})\mathbf{1}_{(-\infty,0)}(x) + \mathbf{1}_{[0,\infty)}(x), & \alpha<0 \quad \text{(Weibull distribution)};\\
\exp(-e^{-x})\mathbf{1}_\R(x), & \alpha=0 \quad \text{(Gumbel distribution)}.
\end{cases}
$$

In 2021, Vargas and Voiculescu (\cite{VV21}) initiated extreme value theory in the context of Boolean independence for non-commutative random variables (see \cite{SW97} for details on Boolean probability theory). Let $(\calH,\xi)$ be a Hilbert space equipped with a unit vector $\xi \in \calH$, and define the vector state $\varphi(X):=\langle T\xi, \xi\rangle_{\calH}$ for a (bounded) linear operator $T$ on $\calH$. If $P$ and $Q$ are Boolean independent projections on $\calH$ with $\varphi(P)=1-p$ and $\varphi(Q)=1-q$ for some $p,q \in [0,1]$, then they showed that $\varphi(P\lor Q)=1-r$ where $r\in [0,1]$ satisfies
$$
r^{-1}-1=(p^{-1}-1) + (q^{-1}+1).
$$
If either $p=0$ or $q=0$, then $r$ is defined to be zero. This identity motivates the definition of a semigroup structure $([0,1], \cup \hspace{-.65em}\lor)$ with the binary operation:
$$
(x\cup \hspace{-.89em}\lor y)^{-1} -1 := (x^{-1}-1) + (y^{-1}-1), \quad x,y \in (0,1]
$$
and $x\cup \hspace{-.89em}\lor y :=0$ if either $x=0$ or $y=0$.

Due to certain constraints in the Boolean setting and operator theory, one must restrict to the class of distribution functions supported on $[0,\infty)$ (see Section 2 in \cite{VV21} for details). Let $\Delta_+$ be the set of all distribution functions on $[0,\infty)$. For $F,G\in \Delta_+$, the {\it Boolean max-convolution} is defined as
$$
(F \cup\hspace{-.89em}\lor G)(x):= F(x) \cup\hspace{-.89em}\lor G(x), \quad x \ge0,
$$
with the $n$-fold Boolean max-convolution defined recursively as
$$
F^{\cup \hspace{-.55em}\lor n}:= \underbrace{F\cup\hspace{-.89em}\lor \cdots \cup\hspace{-.89em}\lor F}_{n \text{ times}}, \qquad n\in \N.
$$

One of the most important results in this theory is the discovery of a semigroup isomorphism between $([0,1],\cup\hspace{-.65em}\lor)$ and $([0,1],\cdot)$, given by the map
$$ 
\calX(u):=\exp\left(1-\frac{1}{u}\right), \quad u \in (0,1], \quad \text{and} \quad \calX(0):=0
$$
which satisfies
$$
\calX(u \cup\hspace{-.89em}\lor v) =\calX(u)\calX(v), \qquad u,v\in [0,1].
$$
The inverse map $\calX^{\langle-1\rangle}$ is given by
$$
\calX^{\langle -1\rangle} (u):= \frac{1}{1-\log u}, \quad u\in (0,1], \quad \text{and} \quad \calX^{\langle -1\rangle}(0):=0
$$
and satisfies
$$
\calX^{\langle -1\rangle} (uv) = \calX^{\langle -1\rangle} (u) \cup \hspace{-.89em}\lor \calX^{\langle -1\rangle} (v), \qquad u,v\in [0,1].
$$
For $F\in \Delta_+$, we define the following distribution functions:
$$
(\calX(F))(x):=\calX(F(x)) \quad \text{and} \quad   (\calX^{\langle-1\rangle}(F))(x):=\calX^{\langle-1\rangle}(F(x)), \quad x\ge0.
$$

The mapping $\calX$ facilitates the transfer of classical results into the Boolean setting. Using this framework, Vargas and Voiculescu proved that the distribution function of the normalized spectral maximum $(X_1\lor \cdots \lor X_n)/a_n$, given by
$$
x\mapsto F^{\cup\hspace{-.55em}\lor n}(a_nx), \qquad x\in [0,\infty)
$$
converges weakly to the {\it Boolean extreme value distribution}, also known as {\it Dagum distribution}:
$$
\Phi_\alpha^{\rm boolean}(x) = \frac{1}{1+x^{-\alpha}}, \quad x>0 \quad \text{and} \quad \alpha>0,
$$
as $n\to\infty$ (see Theorem 4.1 in \cite{VV21}).
Here, $\{X_n\}_n$ is a sequence of Boolean i.i.d. positive random variables with spectral distribution function $F \in \Delta_+$, $\{a_n\}_n \subset (0,\infty)$ is a normalization sequence, and $\lor$ denotes the Ando max-operation (see \cite{A89}, \cite{BV06}, \cite{O71}). The following important relations are known:
$$
\calX(\Phi_\alpha^{\rm boolean}) = \Phi_\alpha  \quad \text{and} \quad \calX^{\langle -1\rangle} (\Phi_\alpha)=\Phi_\alpha^{\rm boolean}, \qquad \alpha>0.
$$

The purpose of this paper is to determine the convergence rate of normalized spectral maximum of Boolean i.i.d. positive random variables to the Boolean extreme value distribution, under the von Mises condition. For $\alpha>0$ and for a distribution function $F$ on $\R$ which is differentiable on a neighborhood of $\infty$, we define
\begin{align}\label{eq:F_k}
k_{\alpha,F}(x)= \frac{xF'(x)}{F(x)(1-F(x))}-\alpha.
\end{align}
In this paper, we focus on the following class of distribution functions.
\begin{defn}\label{def:F_alpha}
Let us consider $\alpha>0$. Denote by $\calV_\alpha$ the set of all distribution functions $F$ on $\R$ such that
\begin{enumerate}[\rm (1)]
\item $F(x)<1$ for all $x\in (0,\infty)$, $F(0)=0$ and $F$ is a $C^1$-function on a neighborhood of $\infty$; \label{condi1}
\item there exists a non-increasing continuous function $g$ such that $g(x) \to 0$ as $x\to\infty$ and 
\begin{align}\label{eq:bound}
|k_{\alpha,F}(x)|\le g(x)
\end{align}
for all $x$ in the neighborhood of $\infty$.  \label{condi2}
\end{enumerate}
In this case, we call the function $g$ the {\it auxiliary function with \eqref{eq:bound}} and the above assumption \eqref{condi2} the {\it von Mises condition}. Moreover, we define $\calV_{\alpha,+}:=\calV_\alpha \cap \Delta_+$.
\end{defn}

\begin{thm}\label{thm:main}
Given $\alpha>0$, let us consider $F\in \calV_{\alpha,+}$ and let $g$ be an auxiliary function with \eqref{eq:bound}. Define two positive sequences $a_n>0$ and $a_n'>0$ by $F(a_n)=e^{-\frac{1}{n}}$ and $F(a_n')=\frac{n}{n+1}$, respectively, and define $A_n:=a_n/a_n'\in (0,1]$. Then we obtain
\begin{align*}
\sup_{x \in \R} |F^{\cup \hspace{-.55em}\lor n}(a_n x)- \Phi_\alpha^{\rm boolean} (x)| \le O \left(g(\rho(a_n)) \lor n^{-1} \lor (A_n^{-1}-1) \right),
\end{align*}
for sufficiently large $n$, where $\rho$ is the inverse function of
$$
t\mapsto t\left\{ \frac{\alpha e}{g(t)} - (e+1)\right\}^{\frac{1}{\alpha -g(t)}}.
$$
\end{thm}

A proof of the above theorem is sketched here. Let us set $d_n (x):= |F^{\cup \hspace{-.55em}\lor n}(a_n x)- \Phi_\alpha^{\rm boolean} (x)|$. Since both  $F$ and $\Phi_\alpha^{\rm boolean}$ are supported on $[0,\infty)$, it follows immediately that $d_n(x)=0$ for all $x\le 0$. For $x\ge 1$, we show that $d_n(x)$ is estimated by the error between the free max-convolution power of the distribution function $x\mapsto F(a_nx)$ and the free extreme value distribution (see Section \ref{sec:FEV} for details on the free max-settings), together with an additional error term of order $n^{-1}$. Using a previous result in \cite{KU25}, we obtain 
$$
\sup_{x\ge 1}d_n(x) \le O(g(a_n)\lor n^{-1}), 
$$
as shown in Lemma \ref{lem:Estimate:freemax} and Proposition \ref{prop:Estimate:booleanmax}. 

Next, we consider the case $0<x<1$. Here, the semigroup isomorphism $\calX$ plays a crucial role. By definition of $\calX$, we have $|h_{\alpha,\calX(F)}|\le g$ on the neighborhood of $\infty$, where the function $h_{\alpha,F}$ is defined in \eqref{eq:F_h}. Then, by (2.45) in \cite{R}, which is a classical extreme value theory estimate, we have
$$
\Phi_{\alpha+g(a_n'x)} (x) \le \calX(F)^n (a_n'x)\le \Phi_{\alpha-g(a_n'x)}(x),
$$
where $F(a_n')=\frac{n}{n+1}$. By applying $\calX^{\langle-1\rangle}$ to both sides of the above inequality, and using Lemma~\ref{lem:Unif_Dagum} to uniformly control the distance between Boolean extreme value (Dagum) distributions with different parameters, together with the monotonicity of $g$, we obtain
\begin{align*}
\sup_{0<x<1}|F^{\cup \hspace{-.55em}\lor n}(a_n'x) - \Phi_\alpha^{\rm boolean}(x)| \le \frac{g(a_nx_n)}{e(\alpha-g(a_nx_n))} \lor \Phi_{\alpha-g(a_nx_n)}^{\rm boolean}(x_n),
\end{align*}
for any sequence $\{x_n\}_n\subset (0,1)$ such that $x_n\to 0$ and $a_n x_n\to \infty$ as $n\to \infty$. To optimize this bound, we choose $\{x_n\}_n$ so that the two terms on the right-hand side are of the same order. Intuitively, the function $\rho$ in Theorem~\ref{thm:main} describes the critical scale at which these two contributions balance each other; see the proof of Proposition~\ref{prop:Estimate:booleanmax2} for details. In the classical extreme value theory setting, it is known that a similar balance between terms of comparable order arises, leading to a function analogous to $\rho$; see (2.47) in \cite{R}. Consequently, choosing $\{x_n\}_n$ according to this balancing condition, we obtain
\begin{align}\label{eq:outline}
\sup_{0<x<1}|F^{\cup \hspace{-.55em}\lor n}(a_n'x) - \Phi_\alpha^{\rm boolean}(x)| \le \frac{g(\rho(a_n))}{e(\alpha-g(\rho(a_n)))}.
\end{align}
Finally, for $0<x<1$, we estimate
$$
d_n(x) \le |F^{\cup \hspace{-.55em}\lor n}(a_n' (A_nx)) - \Phi_\alpha^{\rm boolean}(A_n x)| + |\Phi^{\rm boolean}_\alpha(A_nx) -\Phi_\alpha^{\rm boolean}(x)|,
$$ 
where $A_n=a_n/a_n'\in (0,1]$. Since $|\Phi_\alpha^{\rm boolean}(A_nx)-\Phi_\alpha^{\rm boolean}(x)| \le \alpha(A_n^{-1}-1)$ for all $x\in (0,1)$ and the inequality \eqref{eq:outline} holds, we get 
$$
\sup_{0<x<1} d_n(x) \le O\left(g(\rho(a_n))\lor (A_n^{-1}-1)\right).
$$ 

Combining the estimates for all three regions of $x$ yields the main theorem.

In Section \ref{sec:FEV}, we introduce the concepts of max-convolution and extreme value distributions in free probability theory. In Section \ref{sec:ratefree}, we derive a convergence rate toward the free extreme value distribution (see Lemma \ref{lem:Estimate:freemax}) as a partial result contributing to the proof of Theorem \ref{thm:main}. In Section \ref{sec:proof}, we provide the complete proof of Theorem \ref{thm:main}. In Section \ref{sec:application}, we apply Theorem \ref{thm:main} to several distributions that belong to the Boolean max-domain of attraction of $\Phi_\alpha^{\rm boolean}$; see Propositions \ref{ex:appli} and \ref{ex:appli2}.


\section{Preliminaries in free extreme value theory}

\subsection{Free max-convolution and free extreme value distributions}
\label{sec:FEV}

In 2006, Ben Arous and Voiculescu (\cite{BV06}) established extreme value theory in the framework of free probability theory.  In the years that followed, considerable effort has been directed toward advancing the theory of free extreme values (\cite{BK10},  \cite{BGCD10}, \cite{GN20}, \cite{HU21}, \cite{KU25}, \cite{U22}, \cite{U22-2}, \cite{U21}). We will explain more details as follows.

According to Definition 6.1 in \cite{BV06}, for any distribution functions $F$ and $G$ on $\R$, we define
$$
F \Box\hspace{-.95em}\lor G :=  \max\{F+G-1,0\} \quad \text{and} \quad F^{\Box\hspace{-.58em}\lor  n} := \underbrace{F\Box\hspace{-.95em}\lor  \cdots \Box\hspace{-.95em}\lor  F}_{n\text{ times}}.
$$ The operation $\Box\hspace{-.74em}\lor$ is called the {\it free max-convolution}. By Theorem 6.8 in \cite{BV06}, the limiting distribution of the normalized maximum of free i.i.d non-commutative real random variables is known to be type equivalent to one of the three distributions:
\begin{align*}
\Phi_\alpha^{\rm free}(x) := 
\begin{cases}
 (1-x^{-\alpha})\mathbf{1}_{[1,\infty)}(x), & \alpha>0 \quad \text{(Pareto distribution)};\\
 \{1-(-x)^{-\alpha}\}\mathbf{1}_{[-1,0]}(x) + \mathbf{1}_{(0,\infty)}(x), &  \alpha < 0 \quad \text{(Beta distribution)};\\
 (1-e^{-x})\mathbf{1}_{[0,\infty)}(x) & \alpha=0  \quad \text{(Exponential distribution)}.
 \end{cases}
\end{align*}
The distributions $\Phi_\alpha^{\rm free}$ are called the {\it free extreme value distributions}. Rigorously, for a distribution function $F$ on $\R$, there exist $a_n>0$, $b_n\in \R$ and a non-degenerate distribution function $G$ such that, for any $x\in \R$,
$F^{\Box\hspace{-.57em}\lor n} (a_n x +b_n) \to G(x)$ as $n\to\infty$.
Then the function $G$ is the free extreme value distribution.

\subsection{Rate of convergence for free extreme values}
\label{sec:ratefree}

In this section, we give the rate of convergence toward the free extreme value distribution under the von Mises condition. 
Given $\alpha>0$ and a distribution function $F$ on $\R$ that is differentiable in a neighborhood of $\infty$, we define the function
\begin{align}\label{eq:F_h}
h_{\alpha,F}(x) = \frac{xF'(x)}{F(x)(-\log F(x))}-\alpha.
\end{align}

We can estimate the function $h_{\alpha,F}$ as follows.
\begin{lem}\label{lem:Estimate_h}
Let $\alpha>0$, $F\in \calV_{\alpha}$ and let $g$ be an auxiliary function with \eqref{eq:bound}. Then we obtain
$$
|h_{\alpha,F}(x)|\le g(x) + \alpha \cdot \frac{-\log F(x)-(1-F(x))}{-\log F(x)} = :u(x)
$$
for all $x$ in the neighborhood of $\infty$. Moreover, $u(x)$ is a non-increasing continuous function such that $u(x) \to 0$ as $x\to \infty$.
\end{lem}
\begin{proof}
By \eqref{eq:F_k} and \eqref{eq:F_h}, 
\begin{align*}
h_{\alpha,F}(x) = k_{\alpha,F}(x) \frac{1-F(x)}{-\log F(x)} + \alpha\cdot\frac{1-F(x)-(-\log F(x))}{-\log F(x)}.
\end{align*}
Since $1-F(x) \le -\log F(x)$, it follows from the triangle inequality and Definition \ref{def:F_alpha} that
\begin{align*}
|h_{\alpha,F}(x)| 
&\le |k_{\alpha,F}(x)| + \alpha \cdot \frac{-\log F(x)-(1-F(x))}{-\log F(x)} \\
& \le g(x) + \alpha\cdot  \frac{-\log F(x)-(1-F(x))}{-\log F(x)} =: u(x).
\end{align*}
We set 
$$
\ell (u) := \frac{-\log u - (1-u)}{-\log u}, \quad u\in (0,1).
$$ 
One can see that
\begin{itemize}
\item $\ell'(u) = \dfrac{-u\log u- (1-u)}{u(-\log u)^2}<0$ for all $u\in (0,1)$;
\item $\lim_{u\to 0^+}\ell(u)=1$ and $\lim_{u\to 1^-} \ell(u)=0$.
\end{itemize}
Thus the function 
$$
x\mapsto \frac{-\log F(x)-(1-F(x))}{-\log F(x)}
$$
is non-increasing and goes to $0$ as $x\to \infty$, and therefore so is $u(x)$.
\end{proof}

Choose $a_n = F^{\leftarrow} (e^{-\frac{1}{n}})$, where $F^{\leftarrow} (y) := \inf \{x\in \R: F(x) \ge y\}$. Due to Proposition 2.2 in \cite{KU25}, we have $F^{\Box\hspace{-.57em}\lor n}(a_n\cdot )\xrightarrow{w} \Phi_\alpha^{\rm free}$. Finally, we get the following conclusion.

\begin{lem}\label{lem:Estimate:freemax}
Given $\alpha>0$, let us consider $F\in \calV_{\alpha}$ and an auxiliary function $g$ with \eqref{eq:bound}. We define $a_n=F^{\leftarrow} (e^{-\frac{1}{n}})$. Then $F^{\Box\hspace{-.57em}\lor n}(a_n\cdot )\xrightarrow{w} \Phi_\alpha^{\rm free}$ and
\begin{align*}
\sup_{x\ge 1}|F^{\Box\hspace{-.57em}\lor n}(a_nx) - \Phi_\alpha^{\rm free}(x)| \le O(g(a_n)\lor n^{-1})
\end{align*}
for sufficiently large $n$.
\end{lem}
\begin{proof}
For $x\ge 1$, we have
\begin{align*}
|F^{\Box\hspace{-.57em}\lor n}(a_nx) - \Phi_\alpha^{\rm free}(x)| 
&=|x^{-\alpha}-n(1-F(a_nx))|\\
&\le |x^{-\alpha}- n(-\log F(a_nx))| + n \left\{(-\log F(a_nx))-(1-F(a_nx)) \right\}.
\end{align*}
We define $r(u):=-\log u -(1-u)$ for $u\in (0,1)$. Then $r$ is non-increasing on $(0,1)$, and therefore $r(F(a_nx))\le r(F(a_n))=r(e^{-\frac{1}{n}})$. Hence we get
\begin{align*}
 n \{(-\log F(a_nx))-(1-F(a_nx)) \} &= nr(F(a_nx))\le nr(e^{-\frac{1}{n}})\\
 &= n \{ (-\log e^{-\frac{1}{n}}) -(1-e^{-\frac{1}{n}})\} = O(n^{-1}).
\end{align*}
Since $u(a_n)\to 0$ as $n\to \infty$, for any $\epsilon\in (0,\alpha)$, there exists $n_0\in \N$ such that $|u(a_n)|<\epsilon$ for $n\ge n_0$. For such an $n$, we get
$$
|x^{-\alpha}-n(-\log F(a_nx))| \le \frac{1}{e(\alpha-\epsilon)} u(a_n)
$$
by the proof of Theorem 2.4 in \cite{KU25}. By Lemma \ref{lem:Estimate_h},
\begin{align*}
u(a_n) &= g(a_n) + \alpha (1-n(1-e^{-\frac{1}{n}}))\\
&= g(a_n) + \frac{\alpha}{2n}(1+ O(n^{-1}))  \\
&= O(g(a_n)\lor n^{-1}).
\end{align*}
Hence, the desired result is obtained.
\end{proof}


\section{Proof of main theorem}
\label{sec:proof}

In this section, we prove Theorem \ref{thm:main}. We begin by establishing a partial result toward the main theorem.
\begin{prop}\label{prop:Estimate:booleanmax}
Given $\alpha>0$, let us consider $F\in \calV_{\alpha,+}$ and an auxiliary function $g$ with \eqref{eq:bound}. We define $a_n =F^{\leftarrow}(e^{-\frac{1}{n}})$. Then we obtain
\begin{align*}
\sup_{x\ge 1} |F^{\cup \hspace{-.53em} \lor n} (a_nx)- \Phi_{\alpha}^{\rm boolean}(x)| \le O(g(a_n)\lor n^{-1}).
\end{align*}
\end{prop}
\begin{proof}
For $x\ge 1$, we get
\begin{align*}
|F^{\cup \hspace{-.53em} \lor n} (a_nx)- \Phi_{\alpha}^{\rm boolean}(x)|
&=\left|\frac{F(a_nx)}{n-(n-1)F(a_nx)}-\frac{1}{1+x^{-\alpha}} \right|\\
&\le | F(a_nx)(1+x^{-\alpha}) - n + (n-1)F(a_nx)|\\
& \le |nF(a_nx)-n+x^{-\alpha}| + x^{-\alpha}(1-F(a_nx))\\
&=|F^{\Box\hspace{-.57em}\lor n}(a_nx) - \Phi_\alpha^{\rm free}(x)|  +x^{-\alpha}(1-F(a_nx))\\
& \le O(g(a_n)\lor n^{-1}) + 1-F(a_n)\\
& =  O(g(a_n)\lor n^{-1}),
\end{align*}
where the fifth inequality holds by Lemma \ref{lem:Estimate:freemax}.
\end{proof}

Notice that, for any $F\in \calV_{\alpha,+}$, we obtain
\begin{align*}
h_{\alpha,\calX(F)}(x) &= \frac{\calX(F)'(x)}{\calX(F)(x)(-\log \calX(F(x))}-\alpha \\
&= \frac{F'(x)}{F(x)(1-F(x))}-\alpha \\
&= k_{\alpha,F}(x).
\end{align*}
Thus, by assumption \eqref{condi2} of Definition~\ref{def:F_alpha}, we have $|h_{\alpha,\calX(F)}|\le g$ in a neighborhood of $\infty$. If we choose $a_n'>0$ such that $\calX(F)(a_n') = e^{-\frac{1}{n}}$ (equivalently, $F(a_n')=\frac{n}{n+1}$), then $\calX(F)^n(a_n'\cdot) \xrightarrow{w} \Phi_{\alpha}$; see Page 107 in \cite{R}. Moreover, it was known that, for any $x<1$,
$$
\Phi_{\alpha+ g(a_n'x)}(x) \le \calX(F)^n(a_n'x )  \le \Phi_{\alpha - g(a_n'x)}(x),
$$
due to (2.45) in \cite{R}, which is a standard estimate from classical extreme value theory. Since $\calX^{\langle-1\rangle}$ is strictly increasing on $[0,1]$, we get
\begin{align}\label{eq:keyinequality}
\Phi_{\alpha+ g(a_n'x)}^{\rm boolean}(x)  \le F^{\cup \hspace{-.53em}\lor n}(a_n'x) \le \Phi_{\alpha - g(a_n'x)}^{\rm boolean}(x).
\end{align}

To estimate the distance between $F^{\cup \hspace{-.53em}\lor n}(a_n'x)$ and $\Phi_\alpha^{\rm boolean}$, we derive the following inequality for the Dagum distributions.

\begin{lem}\label{lem:Unif_Dagum}
For any $\alpha_1,\alpha_2>0$, we have 
$$\sup_{0<x<1} |\Phi_{\alpha_1}^{\rm boolean}(x)-\Phi_{\alpha_2}^{\rm boolean}(x)| \le  e^{-1} \ \dfrac{|\alpha_2-\alpha_1|}{\alpha_1 \land \alpha_2}.$$
\end{lem}
\begin{proof}
Without loss of generality, we may assume that $\alpha_1 <\alpha_2$. Note that 
$$
\frac{\partial}{\partial\beta} \frac{1}{1+x^{-\beta}} = \frac{x^\beta \log x}{(1+x^\beta)^2}
$$ 
for $\beta>0$ and $x\in (0,1)$. By the mean value theorem, for each $x\in (0,1)$, there exists $\gamma \in (\alpha_1,\alpha_2)$ such that
$$
|\Phi_{\alpha_1}^{\rm boolean}(x)-\Phi_{\alpha_2}^{\rm boolean}(x)| = \left|\frac{x^\gamma \log x}{(1+x^\beta)^2} \right| (\alpha_2-\alpha_1).
$$
Therefore, we get
\begin{align*}
\sup_{0<x<1} |\Phi_{\alpha_1}^{\rm boolean}(x)-\Phi_{\alpha_2}^{\rm boolean}(x)| 
&= \sup_{0<x<1} \left|\frac{x^\gamma \log x}{(1+x^\gamma)^2} \right| (\alpha_2-\alpha_1)\\
&\le \sup_{0<x<1} \sup_{\beta \in [\alpha_1, \alpha_2]} (-x^{\beta}\log x) (\alpha_2-\alpha_1)\\
&\le (\alpha_2-\alpha_1) \sup_{0<x<1} (-x^{\alpha_1} \log x)=e^{-1}\ \frac{\alpha_2-\alpha_1}{\alpha_1}.
\end{align*}
\end{proof}

\begin{prop}\label{prop:Estimate:booleanmax2}
Given $\alpha>0$, let us consider $F\in \calV_{\alpha,+}$ and an auxiliary function $g$ with \eqref{eq:bound}. Moreover, we define $a_n=F^{\leftarrow}(e^{-\frac{1}{n}})$ and $a_n'=F^{\leftarrow}(\frac{n}{n+1})$. Then
\begin{align*}
\sup_{0<x<1}|F^{\cup \hspace{-.53em}\lor n}(a_n'x)- \Phi_\alpha^{\rm boolean}(x) | \le \frac{g(\rho(a_n))}{e(\alpha-g(\rho(a_n)))}
\end{align*}
for sufficiently large $n$, where the function $\rho$ is the inverse function of
$$
t\mapsto t\left\{ \frac{\alpha e}{g(t)} - (e+1)\right\}^{\frac{1}{\alpha -g(t)}}.
$$
\end{prop}
\begin{proof}
The proof is essentially the same as that of Proposition 2.13 in \cite{R}. By the inequality \eqref{eq:keyinequality}, for any $x\in (0,1)$, we have
$$
\Phi_{\alpha+ g(a_n'x)}^{\rm boolean}(x) -\Phi_{\alpha}^{\rm boolean}(x) \le F^{\cup \hspace{-.53em}\lor n}(a_n'x)-\Phi_{\alpha}^{\rm boolean}(x) \le \Phi_{\alpha - g(a_n'x)}^{\rm boolean}(x)-\Phi_{\alpha}^{\rm boolean}(x).
$$
By Lemma \ref{lem:Unif_Dagum}, we obtain
$$
|F^{\cup \hspace{-.53em}\lor n}(a_n'x)- \Phi_\alpha^{\rm boolean}(x) | \le \frac{g(a_n'x)}{e(\alpha-g(a_n'x))}.
$$
Since $F(a_n')=\frac{n}{n+1} \ge e^{-\frac{1}{n}}=F(a_n)$ and $F$ is a distribution function, we have $a_n' x \ge a_nx$ for each $x\in (0,1)$. The monotonicity of $g$ implies that
$$
\frac{g(a_n'x)}{e(\alpha-g(a_n'x))} \le \frac{g(a_nx)}{e(\alpha-g(a_nx))}.
$$
For any $x\in (0,1)$, we have
$$
|F^{\cup \hspace{-.53em}\lor n}(a_n'x)- \Phi_\alpha^{\rm boolean}(x) | \le \frac{g(a_nx)}{e(\alpha-g(a_nx))},
$$
and hence, for any sequence $\{x_n\}_n \subset (0,1)$ with $x_n\to 0$ and $a_nx_n\to \infty$ as $n\to \infty$, we get
$$
\sup_{x_n \le x <1} |F^{\cup \hspace{-.53em}\lor n}(a_n'x)- \Phi_\alpha^{\rm boolean}(x) | \le \frac{g(a_nx_n)}{e(\alpha-g(a_nx_n))}.
$$
For $x\le x_n$, we obtain
\begin{align*}
F^{\cup \hspace{-.53em}\lor n}(a_n'x) 
\le F^{\cup \hspace{-.53em}\lor n}(a_n'x_n)\le \Phi_{\alpha-g(a_n'x_n)}^{\rm boolean}(x_n) \le \Phi_{\alpha-g(a_nx_n)}^{\rm boolean}(x_n)
\end{align*}
and $\Phi_\alpha^{\rm boolean}(x_n) \le \Phi_{\alpha-g(a_nx_n)}^{\rm boolean}(x_n)$. Thus, we get
$$
\sup_{0<x<1}|F^{\cup \hspace{-.53em}\lor n}(a_n'x)- \Phi_\alpha^{\rm boolean}(x) | \le\frac{g(a_nx_n)}{e(\alpha-g(a_nx_n))} \lor \Phi_{\alpha-g(a_nx_n)}^{\rm boolean}(x_n).
$$
The choice of $x_n$ to minimize the right hand side is to satisfy
\begin{align}\label{eq:g=b}
\frac{g(a_nx_n)}{e(\alpha-g(a_nx_n))}=\Phi_{\alpha-g(a_nx_n)}^{\rm boolean}(x_n).
\end{align}
We explain how to choose such a $x_n$ as follows. Let us set 
$$
a(t):=\left(\frac{1}{-\log F}\right)^{\leftarrow}(t) = \inf \left\{u : \frac{1}{-\log F(u)}\ge t \right\}
$$ 
and define a non-decreasing function $\rho(t)$ via the relation $\rho(a(t))=a(t)x(t)$, where $x(t)$ is an unknown decreasing function such that $x(t)\to 0$ and $a(t)x(t) \to \infty$ as $t\to\infty$. We assume that $\rho(t)$ satisfies
$$
\frac{g(\rho(a(t)))}{e(\alpha-g(\rho(a(t))))} = \frac{1}{1+x(t)^{-(\alpha-g(\rho(a(t))))}}.
$$
This implies that, if let $\rho^{\leftarrow}(t)$ be the inverse function of $\rho(t)$ then it should be
\begin{align}\label{eq:rhoinversedef}
\rho^{\leftarrow}(t) = t\left\{ \frac{\alpha e}{g(t)} - (e+1)\right\}^{\frac{1}{\alpha -g(t)}}
\end{align}
for sufficiently large $t$ such that $g(t) < \alpha e(e+1)^{-1}$. Thus, one can see
\begin{align*}
\lim_{t\to \infty} \frac{\rho^{\leftarrow}(t)}{t}=\infty, \quad
\text{and hence} \quad
\lim_{t\to \infty} \frac{\rho (t)}{t}=0.
\end{align*}
Conversely, if we define the function $\rho^\leftarrow(t)$ by \eqref{eq:rhoinversedef} and consider $\rho(t)$ as the inverse function of $\rho^\leftarrow(t)$, then $x(t) = \rho(a(t))/a(t)$ satisfies $x(t)\to 0$ and $a(t)x(t) =\rho(a(t))\to \infty$ as $t\to\infty$. Therefore $x_n: = x(n)$ satisfies \eqref{eq:g=b}, and also we have
\begin{align*}
\sup_{0<x<1}|F^{\cup \hspace{-.53em}\lor n}(a_n'x)- \Phi_\alpha^{\rm boolean}(x) | \le\frac{g(\rho(a_n))}{e(\alpha-g(\rho(a_n)))}.
\end{align*}
\end{proof}

Finally, we prove Theorem \ref{thm:main} as follows.

\begin{proof}
One can easily see that $|F^{\cup \hspace{-.53em}\lor n}(a_n x)- \Phi_\alpha^{\rm boolean} (x)|=0$ for $x\le 0$. Since $a_n' \ge a_n>0$, we have $A_n=a_n/a_n'\in (0,1]$. Furthermore, we get $A_n\to 1$ as $n\to \infty$. Hence, for any $0<x<1$,
\begin{align*}
|F^{\cup \hspace{-.53em}\lor n}(a_n x)&- \Phi_\alpha^{\rm boolean} (x)| \\
&= |F^{\cup \hspace{-.53em}\lor n}(a_n' (A_nx))- \Phi_\alpha^{\rm boolean} (x)|\\
&\le |F^{\cup \hspace{-.53em}\lor n}(a_n' (A_nx))- \Phi_\alpha^{\rm boolean} (A_nx)| + |\Phi_\alpha^{\rm boolean} (A_nx)-\Phi_\alpha^{\rm boolean} (x)|\\
&\le \frac{g(\rho(a_n))}{e(\alpha-g(\rho(a_n)))} +  |\Phi_\alpha^{\rm boolean} (A_nx)-\Phi_\alpha^{\rm boolean} (x)|
\end{align*}
by Proposition \ref{prop:Estimate:booleanmax2}. Since $\rho(a_n)\to \infty$, for any $0<\epsilon<\alpha$, there exists $n_0 \in \N$ such that $|g(\rho(a_n))|<\epsilon$ for $n\ge n_0$. Hence
$$
\frac{g(\rho(a_n))}{e(\alpha-g(\rho(a_n)))} = O(g(\rho(a_n))).
$$

On the other hand, it is easy to see
\begin{align*}
|\Phi_\alpha^{\rm boolean} (A_nx)-\Phi_\alpha^{\rm boolean} (x)| \le  \alpha \left(A_n^{-1}-1\right), \quad 0<x<1.
\end{align*}
Thus, 
$$
\sup_{0<x<1}|F^{\cup \hspace{-.53em}\lor n}(a_n x)- \Phi_\alpha^{\rm boolean} (x)|  \le O\left( g(\rho(a_n))\lor (A_n^{-1}-1)\right).
$$
Since $\frac{\rho(a_n)}{a_n}\to 0$ as $n\to \infty$, we have $g(\rho(a_n))\ge g(a_n)$ for sufficiently large $n$. Combining Proposition \ref{prop:Estimate:booleanmax} and the above estimates, we get the desired formula.
\end{proof}

\section{Applications of main result}
\label{sec:application}

In this section, we apply Theorem \ref{thm:main} to several distributions that belong to the Boolean max-domain of attraction of $\Phi_\alpha^{\rm boolean}$.

\subsection{Classical Fr\'{e}chet distributions}
Let $\alpha>0$. Recall that $\Phi_\alpha$ denotes the classical Fréchet distribution, that is,
$$
\Phi_\alpha(x)=\exp(-x^{-\alpha})\mathbf{1}_{(0,\infty)}(x).
$$
\begin{prop}\label{ex:appli}
Given $\alpha>0$, let us set $a_n=n^{\frac{1}{\alpha}}$. Then, for sufficiently large $n$, we have
\begin{align}\label{prop:asmpt}
\sup_{x\in \R} |\Phi_\alpha^{\cup \hspace{-.53em}\lor n}(a_n x) - \Phi_\alpha^{\rm boolean}(x)| \le O\left(\frac{1}{\sqrt{n}}\right).
\end{align}
\end{prop}
\begin{proof}
First, we define $a_n'>0$ by $\Phi_\alpha(a_n')=\frac{n}{n+1}$. It is easy to see that $a_n'=\left(\log (1+n^{-1})\right)^{\frac{1}{\alpha}}$. We obtain
$$
A_n= \left\{ n\log\left(1+\frac{1}{n}\right)\right\}^{\frac{1}{\alpha}} = 1 -\frac{1}{2\alpha n} + O\left(\frac{1}{n^2}\right),
$$
and therefore
$$
A_n^{-1}-1 = \frac{1}{2\alpha n}\left(1+O\left(\frac{1}{n}\right)\right), \qquad n\to \infty.
$$

Next, we compute the function $k_{\alpha, \Phi_\alpha}$ and define an auxiliary function $g$ with \eqref{eq:bound} as follows:
\begin{align*}
k_{\alpha, \Phi_\alpha}(x) = \alpha \left\{ \frac{x^{-\alpha}}{1-\exp(-x^{-\alpha})}-1 \right\} \quad \text{and} \quad g(x) :=\frac{\alpha}{x^\alpha-1}, \quad \text{respectively}.
\end{align*}

In this case, for sufficiently large $t$, we can compute the asymptotic expansion for the function
\begin{align}\label{eqrhoinverse}
t \left\{\frac{\alpha e}{g(t)} - (e+1) \right\}^{\frac{1}{\alpha-g(t)}} = t \{et^\alpha- (2e+1)\}^{\frac{t^\alpha-1}{\alpha t^\alpha -2\alpha}}.
\end{align}
One can see that
$$
\frac{t^\alpha-1}{\alpha t^\alpha -2\alpha} = \frac{1}{\alpha} (1+O(t^{-\alpha})), \qquad t \to \infty.
$$
Moreover, we get
\begin{align*}
\log\{et^\alpha- (2e+1) \} = 1+ \alpha \log t + O(t^{-\alpha}), \qquad t \to \infty.
\end{align*}
Therefore, 
\begin{align*}
\{et^\alpha &- (2e+1)\}^{\frac{t^\alpha-1}{\alpha t^\alpha -2\alpha}} \\
&= \exp\left\{\frac{1}{\alpha}\left(1+O(t^{-\alpha}) \right)\left(1+ \alpha \log t + O(t^{-\alpha}) \right) \right\}\\
&= \exp\left\{\frac{1}{\alpha} \left( 1+ \alpha \log t + O(t^{-\alpha}\log t) \right) \right\}\\
&=e^{\frac{1}{\alpha}}t (1+O( t^{-\alpha} \log t )),
\end{align*}
as $t \to \infty$. Thus, we can conclude that
\begin{align}\label{eq:asympt_rhoinverse}
\eqref{eqrhoinverse} =e^{\frac{1}{\alpha}}t^2 (1+O( t^{-\alpha} \log t )), \quad t\to \infty.
\end{align}

Note that, the inverse function $\rho$ of \eqref{eqrhoinverse} exists for sufficiently large $t$, and the asymptotic expansion for $\rho$ is given by
\begin{align*}
\rho (x) = e^{-\frac{1}{2\alpha}}x^{\frac{1}{2}} (1+O(x^{-\frac{\alpha}{2}} \log x)), \qquad x\to \infty,
\end{align*}
due to the formula \eqref{eq:asympt_rhoinverse}. The above formula implies that
$$
\rho (a_n) = e^{-\frac{1}{2\alpha}}n^{\frac{1}{2\alpha}} (1+ O (n^{-\frac{1}{2}}\log n)).
$$
Finally, we get
\begin{align*}
g(\rho(a_n)) = \frac{\alpha}{\rho(a_n)^\alpha -1} =\alpha e^{\frac{1}{2}} n^{-\frac{1}{2}} (1+ O(n^{-\frac{1}{2}}\log n)).
\end{align*}

Since
$$
g(\rho (a_n)) \lor n^{-1} \lor (A_n^{-1}-1) = g(\rho(a_n)) = \alpha e^{\frac{1}{2}} n^{-\frac{1}{2}} (1+ O(n^{-\frac{1}{2}}\log n))
$$
for sufficiently large $n$, the desired result \eqref{prop:asmpt} is obtained, due to Theorem \ref{thm:main}.
\end{proof}

\begin{rem}
Our main result is useful for deriving convergence rates of Boolean max-convolution in a fairly general setting. However, it does not address the optimality of these rates. In fact, the estimate given in Proposition \ref{ex:appli} is not optimal. Indeed, a direct computation shows that
\begin{align*}
\sup_{x\in \R} |\Phi_\alpha^{\cup \hspace{-.53em}\lor n}(a_n x) - \Phi_\alpha^{\rm boolean}(x)| = O\left(\frac{1}{n}\right).
\end{align*}
\end{rem}

\subsection{Inverse gamma distributions}

In this section, we consider inverse gamma distributions. For $\alpha>0$, define the upper and lower incomplete gamma functions by
$$
\Gamma(\alpha,z) = \int_z^\infty t^{\alpha-1}e^{-t}dt \quad \text{and} \quad \gamma(\alpha,z) = \int_0^z t^{\alpha-1}e^{-t}dt,
$$
for all $z>0$. Note that, $\gamma(\alpha,z) = \Gamma(\alpha)- \Gamma(\alpha,z)$ for all $z>0$. Let $G_\alpha$ be the distribution function of the {\it inverse gamma distribution} defined by
$$
G_\alpha(x) := \frac{\Gamma(\alpha,1/x)}{\Gamma(\alpha)}, \qquad x>0.
$$
Furthermore, let $a_n>0$ be defined by $G_\alpha(a_n) = e^{-\frac{1}{n}}$ for each $n\in \mathbb{N}$. Although $a_n$ cannot be expressed explicitly, its asymptotic behavior as $n\to\infty$ can be derived. Since
$$
\gamma(\alpha, z) = \frac{z^\alpha}{\alpha} (1+ O(z)), \qquad z\to 0^+,
$$
we obtain
\begin{align}\label{eq:1-G_a}
1-G_\alpha(x) = \frac{\gamma(\alpha,1/x)}{\Gamma(\alpha)} = \frac{x^{-\alpha}}{\Gamma(\alpha+1)}(1+O(x^{-1})), \quad x\to\infty.
\end{align}
If we put $u=1-G_\alpha(x)$, then \eqref{eq:1-G_a} implies that
\begin{align}\label{eq:x-u}
x = \left( \frac{1}{\Gamma(\alpha+1)u}\right)^{\frac{1}{\alpha}} (1+ O(u^{\frac{1}{\alpha}})).
\end{align}
Define 
$$
u_n: = 1-G_\alpha(a_n) = 1-e^{-\frac{1}{n}} = \frac{1}{n}(1+ O (n^{-1})).
$$
Substituting $x=a_n$ (i.e. $u=u_n$) into \eqref{eq:x-u}, we obtain
\begin{align*}
a_n &= \left( \frac{1}{\Gamma(\alpha+1)u_n}\right)^{\frac{1}{\alpha}} (1+ O(u_n^{\frac{1}{\alpha}})) \\
&=\left(\frac{n}{\Gamma(\alpha+1)} \right)^{\frac{1}{\alpha}} (1+O(n^{-1}))^{-\frac{1}{\alpha}} ( 1+ O(n^{-\frac{1}{\alpha}}))\\
&=\left(\frac{n}{\Gamma(\alpha+1)} \right)^{\frac{1}{\alpha}} (1+ O(n^{-(1\land \frac{1}{\alpha})})).
\end{align*}

\begin{prop}\label{ex:appli2}
Let $\alpha>0$. Then, for sufficiently large $n$, we have
\begin{align}\label{prop:invgamma}
\sup_{x\in \R} |G_\alpha^{\cup \hspace{-.53em}\lor n}(a_n x) - \Phi_\alpha^{\rm boolean}(x)| \le O\left(n^{-( \frac{1}{2} \land \frac{1}{1+\alpha})} \right).
\end{align}
\end{prop}
\begin{proof}
Define $a_n'>0$ by $G_\alpha (a_n') = \frac{n}{n+1}$ for each $n\in \N$. A similar argument as above yields
$$
a_n'= \left(\frac{n+1}{\Gamma(\alpha+1)} \right)^{\frac{1}{\alpha}} + O\left(n^{-\frac{1}{\alpha}} \right).
$$
Hence,
\begin{align}\label{eq:An-1}
A_n^{-1}-1 = \frac{a_n'}{a_n}-1 = O\left( n^{-(1\land \frac{1}{\alpha})} \right), \qquad n\to \infty.
\end{align}

Next, we construct an auxiliary function $g$ satisfying \eqref{eq:bound} in the case $F=G_\alpha$. It is easy to verify that, for all $z>0$,
$$
\alpha e^{-z} \le \frac{z^\alpha e^{-z}}{\gamma(\alpha,z)} \le \alpha,
$$
and hence
\begin{align}\label{eq:Gamma1}
\left|\frac{z^\alpha e^{-z}}{\gamma(\alpha, z)} -\alpha  \right| \le \alpha(1-e^{-z}) \le \alpha z.
\end{align}
Moreover,
\begin{align}\label{eq:Gamma2}
1-G_\alpha (x)= \frac{\gamma(\alpha,1/x)}{\Gamma(\alpha)} \le \frac{1}{\alpha \Gamma(\alpha)} x^{-\alpha}, \qquad x>0.
\end{align}
Since
\begin{align*}
k_{\alpha, G_\alpha}(x)=\frac{1}{G_\alpha(x)} \left\{ \left(\frac{x^{-\alpha} e^{-1/x}}{\gamma(\alpha, 1/x)} -\alpha \right)+ \alpha  (1-G_\alpha(x))\right\},
\end{align*}
it follows from \eqref{eq:Gamma1} and \eqref{eq:Gamma2} that
\begin{align*}
|k_{\alpha,G_\alpha}(x)| \le \frac{1}{G_\alpha(x)} \left( \frac{\alpha}{x} + \frac{1}{\Gamma(\alpha)}\cdot \frac{1}{x^\alpha} \right), \qquad x>0.
\end{align*}
Therefore, for sufficiently large $x$, there exists a constant $M>0$ such that
\begin{align}\label{eq:g_gamma}
|k_{\alpha,G_\alpha}(x) | \le M \left(\frac{1}{x} + \frac{1}{x^\alpha}\right)=:g(x).
\end{align}

In the following, we distinguish the cases $\alpha\neq 1$ and $\alpha=1$.
\begin{enumerate}
\item[\rm (i)] Let $\alpha\neq 1$.
For sufficiently large $t$, we derive an asymptotic expansion of the function
\begin{align}\label{eq:Gamma_rho}
t\left\{ \frac{\alpha e}{g(t)}- (e+1) \right\}^{\frac{1}{\alpha-g(t)}} = t \left\{ \frac{\alpha e }{M(t^{-1}+t^{-\alpha})} - (e +1) \right\}^{\frac{1}{\alpha - M(t^{-1}+t^{-\alpha})}}.
\end{align}
Observe that
$$
\frac{1}{\alpha - M(t^{-1}+t^{-\alpha})} = \frac{1}{\alpha}\left(1 + O(t^{-(1\land \alpha)})\right), \qquad t\to \infty.
$$
Moreover, 
\begin{align*}
\log & \left\{ \frac{\alpha e }{M(t^{-1}+t^{-\alpha})}  - (e +1) \right\}\\
& =  1+\log \frac{\alpha}{M} + (1\land \alpha) \log t + O\left( t^{-(1\land \alpha \land |1-\alpha|)}\right), \qquad t\to\infty.
\end{align*}
Combining these estimates, we obtain
\begin{align*}
&\left\{ \frac{\alpha e }{M(t^{-1}+t^{-\alpha})} - (e +1) \right\}^{\frac{1}{\alpha - M(t^{-1}+t^{-\alpha})}}\\
&\ = \exp\left\{ \frac{1}{\alpha} \left(1+\log \frac{\alpha}{M} + (1\land \alpha) \log t + O\left( t^{-(1\land \alpha \land |1-\alpha|)}\right)\right) \left(1 + O(t^{-(1\land \alpha)})\right) \right\}\\
&\ =\exp\left\{ \frac{1}{\alpha} \left( M'+ (1\land \alpha)\log t +O\left( t^{-(1\land \alpha \land |1-\alpha|)} \log t\right) \right) \right\}\\
&\ = e^{\frac{M'}{\alpha}} t^{1\land \frac{1}{\alpha}}\left( 1+ O\left( t^{-(1\land \alpha \land |1-\alpha|)} \log t\right)\right),
\end{align*}
where $M'=1+\log\frac{\alpha}{M}$. Consequently,
\begin{align}\label{eq:Gamma_rho2}
\eqref{eq:Gamma_rho} = e^{\frac{M'}{\alpha}} t^{2\land (1+\frac{1}{\alpha})}\left( 1+ O\left( t^{-(1\land \alpha \land |1-\alpha|)} \log t\right)\right), \qquad t\to \infty.
\end{align}
Note that the inverse function $\rho$ of \eqref{eq:Gamma_rho} exists for sufficiently large $t$. Hence, by \eqref{eq:Gamma_rho2}, we obtain the asymptotic expansion
\begin{align*}
\rho (x) = M'' x^{\frac{1}{2\land (1+1/\alpha)}} \left(1+O\left(x^{-\frac{1\land \alpha \land |1-\alpha|}{2\land (1+1/\alpha)}} \log x \right) \right), \quad x\to\infty,
\end{align*}
where $M''=e^{-\frac{M'}{\alpha(2\land (1+1/\alpha))}}$. This implies that
$$
\rho(a_n)=L n^{\frac{1}{2\alpha \land (1+\alpha)}} \left( 1+ O(n^{-(1\land \alpha)}) +O\left( n^{-\frac{1\land \alpha \land |1-\alpha|}{2\alpha \land (1+\alpha)}} \log n\right)\right).
$$
for some constant $L>0$. In view of \eqref{eq:g_gamma}, we obtain
\begin{align}\label{eq:g(rho)_gamma}
g(\rho(a_n))=O\left(n^{-\frac{1\land \alpha}{2\alpha \land (1+\alpha)}} \right) = O\left(n^{-( \frac{1}{2} \land \frac{1}{1+\alpha})} \right).
\end{align}

It follows from \eqref{eq:An-1}, \eqref{eq:g(rho)_gamma} and Theorem \ref{thm:main} that
$$
g(\rho (a_n)) \lor n^{-1} \lor (A_n^{-1}-1) = O\left(n^{-( \frac{1}{2} \land \frac{1}{1+\alpha})} \right)
$$
for sufficiently large $n$.

\item[\rm (ii)] Let $\alpha=1$. In this case, a similar computation yields
\begin{align*}
\eqref{eq:Gamma_rho} = \frac{e}{2M} t^2 \left( 1+ O\left(t^{-1}\log t \right)\right), \qquad t\to\infty
\end{align*}
and hence
\begin{align*}
\rho(x) = \sqrt{\frac{2M}{e}} x^{\frac{1}{2}} (1+O(x^{-\frac{1}{2}}\log x)), \quad x\to\infty.
\end{align*}
Since $a_n=n(1+O(n^{-1}))$, we obtain
\begin{align*}
\rho(a_n)= \sqrt{\frac{2M}{e}} n^{\frac{1}{2}}(1+ O(n^{-\frac{1}{2}}\log n)).
\end{align*}
This implies that
$$
g(\rho(a_n))=\sqrt{2Me} n^{-\frac{1}{2}} ( 1+O(n^{-\frac{1}{2}}\log n)) = O(n^{-\frac{1}{2}}).
$$
Therefore, $g(\rho (a_n)) \lor n^{-1} \lor (A_n^{-1}-1) = O(n^{-\frac{1}{2}})$ for sufficiently large $n$.
\end{enumerate}

Consequently, the conclusion \eqref{prop:invgamma} holds for all $\alpha>0$.
\end{proof}
\vspace{8mm}

\section*{Acknowledgment}
The author is sincerely grateful to the referee for a careful and thorough reading of the manuscript, as well as for valuable comments and suggestions. The author was supported by JSPS Grant-in-Aid for Young Scientists 22K13925.

\vspace{8mm}

\hspace{-6mm}{\bf Yuki Ueda}\\
Faculty of Education, Department of Mathematics, Hokkaido University of Education, 9 Hokumon-cho, Asahikawa, Hokkaido 070-8621, Japan\\
E-mail: ueda.yuki@a.hokkyodai.ac.jp

\end{document}